\documentclass[leqno,11pt]{article} 
\usepackage{amsmath, amssymb}
\usepackage[dvipsnames]{xcolor}
\usepackage{amsthm} 
\usepackage{amscd}
\usepackage[dvipsnames]{xcolor}
\usepackage[normalem]{ulem} 

\usepackage{verbatim}
\usepackage{multicol,lineno}
\usepackage{enumitem}
\theoremstyle{plain} 
\newtheorem{theorem}{\indent\sc Theorem}[section]
\newtheorem{lemma}[theorem]{\indent\sc Lemma}

\newtheorem{proposition}[theorem]{\indent\sc Proposition}
\theoremstyle{definition} 
\newtheorem{definition}{\indent\sc Definition}

\newtheorem{notation}[theorem]{\indent\sc Notation}

\newcommand{\A}{\widetilde{A}}
\newcommand{\B}{\widetilde{B}}
\newcommand{\C}{\mathbb{C}}
\newcommand{\R}{\mathbb{R}}
\newcommand{\Q}{\mathbb{Q}}
\newcommand{\Z}{\mathbb{Z}}
\newcommand{\N}{\mathbb{N}}
\newcommand{\VV}{\mathcal{V}^0}

\def\deg{{\rm{deg}\ }}

\def\ld{\ldots}

\newcommand{\lam}{\lambda}

\newcommand{\mca}{\mathcal}

\def\ld{\ldots}

\newcommand\AND{\quad\textrm{and}\quad}

\newcommand\CC{\mathbb{C}}

\newcommand\NN{\mathbb{N}}

\newcommand\QQ{\mathbb{Q}}

\newcommand\V{\widetilde{V}}

\newtheorem*{remark0}{\indent\sc Remark}

\newcommand{\keywords}{{
  \footnotesize
  \textbf{Keywords}: $S$-unit equation, Pad\'e approximation, Binomial function, Hypergeometric function}}



\title{$S$-unit equation in two variables and Pad\'{e} approximations}
\author{\textsc{Noriko Hirata-Kohno}, \textsc{Makoto Kawashima}, \textsc{Anthony Po\"{e}ls}  and \textsc{Yukiko Washio}}

\date{} 

\begin{document}

\maketitle

\begin{abstract}
In this article, we use Pad\'{e} approximations, constructed in \cite {KP} for binomial functions, to give a new upper bound for the number of the solutions of the $S$-unit equation. Combining explicit formulae of these Pad\'{e} approximants with a simple argument relying on Mahler measure and on the local height, we refine the bound due to J.~-H. Evertse \cite{E}.
\end{abstract}

\keywords
\let\thefootnote\relax

\section{Introduction and main result}

The main purpose of the present article is to use explicit construction of Pad\'{e} approximations, derived from those constructed in \cite{KP} for binomial functions, in order to improve the upper bound for the number of the solutions of the $S$-unit equation in two variables.

\smallskip

Let $K$ be a number field with degree $m:=[K:\Q]<\infty$ and fix non-zero elements $\lam, \mu\in K$. Consider a finite set $S$ of places of $K$, which contains all the archimedean ones. Denote by $s$ its cardinality and by $U_S$ the set of the $S$-units in $K$. It is known that $U_S$ is a finitely generated abelian group of rank $s-1$.

\smallskip

We deal with the $S$-unit equation
\begin{align}\label{unitequation}
    \lam x+\mu y=1 \,\,\text{~in ~unknowns}~ x, y \in U_S.
\end{align}

The finiteness of the number of solutions of Eq.~\eqref{unitequation} dates back to C.~L.~Siegel  \cite{Siegel}, who implicitly reduced the problem to a system consisting of a finite number of Thue equations.

In 1984,  J.~-H. Evertse proved that the equation $\lam x+\mu y=1$ has at most  $3\times 7^{m+2s}$ solutions  \cite[Theorem 1]{E}.
His proof relies on quantitative arguments on Diophantine inequalities of Roth type, with certain Pad\'{e} approximations. Up to now this result has been the best one.

In a more general setting, F.~Beukers and H.~P.~Schlickewei \cite{BeukersSchlickewei} considered a finitely generated subgroup $\Gamma$ of rank $r$ in
$\left(K\setminus\{0\}\right)^2$ to show that the equation $x+y=1$ in  $(x, y) \in \Gamma$ has at most  $256^{r+1}$ solutions. They used  Pad\'{e} approximations for hypergeometric functions, whereas E.~Bombieri and  W.~Gubler
\cite[Chap.~5]{BombieriGubler} followed a different approach in a geometrical framework
(see also \cite{BombieriChoi} \cite{EG2022}).

It is known that the theory of the linear forms in logarithms gives effective bounds for the height of the solutions in the two variables case ({\it confer} \cite{EG2022}). Nevertheless, in this paper, we restrict ourselves to counting the solutions of the $S$-unit equation.

\smallskip

We prove the following statement.

\begin{theorem}\label{main}

Let  $\lam, \mu\in K\setminus\{0\}$. Then the equation  Eq.~\eqref{unitequation}
has at most
\begin{align}\label{theoremmain}
    \left(3.1+5\left(3.4\right)^{m}\right)\cdot  45^s
\end{align}
solutions $(x, y) \in U_S^2$.
\end{theorem}

\

More precisely, the number of solutions is bounded from above by
\begin{align}\label{winner}
    \min\Big\{\Big(2.81864 \cdot \left(46.8312\right)^s+5\left(3.22803\right)^{m}\cdot 47^s \Big),
     \,\Big(3.06759\cdot \left(44.9866\right)^s + 5\left(3.36406\right)^{m}\cdot 45^s \Big) \Big\}.
\end{align}

In  \cite[Theorem 1]{E}, Evertse established the following upper bound:
\[
    3\cdot  7^{m}\cdot  49^{s}.
\]
Note that a more precise bound  is shown at the page 583, line -8 of \cite{E}:
\begin{align}\label{evertseprecise}
    \Big(2+5\cdot \left(
    3.26396\right)^{m}\Big)\cdot  49^s \leq \Big(2+5\cdot \left(2e^{24/49}\right)^{m}\Big)\cdot  49^s.
\end{align}
Our numerical bound Eq.~\eqref{winner} is better than Eq.~\eqref{evertseprecise} for any positive integers $m$, $s$.

\smallskip

The structure of our proof is similar to that of Evertse, however, we obtain significant improvement, especially for the estimates coming from the Pad\'{e} approximation part. See Lemma~\ref{lem:8} which replaces \cite[Lemma $8$]{E}. For that purpose, starting with constructions from \cite[\S $9$]{KP}, we first establish new explicit identities for Pad\'{e} approximants related to binomial functions (see Sections \ref{section: Pade} and \ref{section: remainder}). In Section~\ref{section: Mahler}, by using an argument relying on Mahler measure we are able to estimate precisely the coefficients of the polynomials involved, both in the arichimedean and the non-arichimedean cases. This leads us in Section~\ref{section:height} to several new inequalities needed for Lemma~\ref{lem:8} (the main result of Section~\ref{qtt}). In the proof of this lemma, we use the local height as much as possible instead of rough upper bounds, which allows us to refine inequalities in \cite{E}. Section~\ref{section: numerical opt} is devoted to some numerical computation for a certain choice of parameters, which leads us to the final estimate Eq.~\eqref{winner}.

\section{Pad\'{e} approximants for a cubic binomial function}
\label{section: Pade}

Throughout the article, we denote by $\NN$  the set of non-negative rational integers and by $\Z_{\leq 0}$ the set of non-positive rational integers.

Given $\ell\in \N$, we write $(1/z^{\ell})$ the ideal of the ring of Laurent series $\Q[[1/z]]$, generated by $1/z^{\ell}$.

\smallskip

Consider rational parameters $a, b, c\in \Q$ with $c\notin \Z_{\leq 0}$.  Let us recall Gauss' hypergeometric function
${}_2F_{1}$ defined by
\begin{eqnarray*}
{}_2F_{1}\biggl(\begin{matrix} a, \, b\\ c\end{matrix} \biggm|z\biggr)
=\displaystyle\sum_{k=0}^\infty\frac{(a)_k(b)_k}{(c)_k}\,\cdot\frac{z^k}{k!}
\end{eqnarray*}
where $(a)_k=\Gamma(a+k)/\Gamma(a)$ is the $k$-th Pochhammer symbol. 
The radius of convergence is $1$, unless $a \in \Z_{\leq 0}$ or $b \in \Z_{\leq 0}$, 
 in which case the function is a polynomial in $z$. It is known that we can extend the domain of definition of  ${}_{2}F_{1}$ by analytic continuation (\textit{confer} \cite{BeukersBirk} \cite{Erd} for details).

 \smallskip

A cubic binomial function is a special case of Gauss' hypergeometric function. It belongs to the class of functions studied in \cite{KP}, for which we construct explicit Pad\'{e} approximants.

\smallskip

The next result can be deduced directly from Proposition $3.5$ and Corollary $3.6$ of \cite{KP}. Note that it is in general difficult to explicitly construct Pad\'e approximants $(P_1(z),P_2(z))$, whereas the existence is obvious.

\begin{lemma} \label{binom Pade}
Consider the cubic binomial function given by the Laurent series
\[
    f(z)=\dfrac{1}{z}\left(1-\dfrac{1}{z}\right)^{1/3}=\sum_{k=0}^{\infty}\dfrac{(-1/3)_k}{k!}\dfrac{1}{z^{k+1}}\in \big(1/z\big)\cdot \Q[[1/z]].
\]
For each $n$,
define the polynomials $P_{n,0}(z),P_{n,1}(z) \in \mathbb{Q}[z]$ by
\begin{align*}
        &P_{n,0}(z)=\sum_{k=0}^n(-1)^{n-k}\binom{n+k-1}{k}\binom{n-4/3}{n-k}z^k\,, \\
        &P_{n,1}(z)=\sum_{k=0}^{n-1}(-1)^{n-1-k}\binom{n+k}{k}\binom{n+1/3}{n-1-k}z^k\,
        &\kern-30pt (\text {with the convention that~}P_{0,1}(z)=0),
    \end{align*}
and set
\begin{align} \label{Rn}
    R_n(z)=P_{n,0}(z)f(z)-P_{n,1}(z)\,.
\end{align}
Then the pair of polynomials $(P_{n,0}, P_{n,1})$ forms Pad\'{e} approximants, and $R_n(z)$ is a Pad\'{e} approximation of $f(z)$, namely,

\begin{align*}
    R_n(z) \in \left(1/z^{n+1}\right)\cdot\QQ[[1/z]].
\end{align*}

Explicitly, we have
\begin{align} \label{approximation}
    R_n(z)\,= \sum_{k=n}^{\infty} \binom{k}{n}\frac{(-1/3)_k(4/3)_n}{(n+k)!} \dfrac{1}{z^{k+1}}=\dfrac{(4/3)_n(-1/3)_n}{(2n)!z^{n+1}}\cdot {}_{2}F_{1} \biggl(\begin{matrix} n+1,n-1/3 ~\\ 2n+1 \end{matrix} \biggm| \dfrac{1}{z}\biggr)\in (1/z^{n+1}) \,.
\end{align}
\end{lemma}


Note that our approximants $(P_{n,0}, P_{n,1})$ are the same as those due to M.~A.~Bennett in \cite[Eq. (5) and Eq. (6)]{BennettAus}  (with the parameters $a=-1$ and $\ell=1$), although the formulaes are different.

\smallskip

Next we establish several identities for the above approximants.

\begin{proposition} \label{A,B}
    Define the polynomials  $A_n(1-z)=z^{n-1}P_{n,1}(1/z)$ and $B_n(1-z)=z^nP_{n,0}(1/z)$. Then, we have
    \begin{align*}
        &A_n(1-z)=\sum_{\ell=0}^{n-1}\binom{n+1/3}{\ell}\binom{n-4/3}{n-1-\ell}(1-z)^{\ell}\,,\\
        &B_n(1-z)=\sum_{\ell=0}^n\binom{n-4/3}{\ell}\binom{n+1/3}{n-\ell}(1-z)^{\ell}\,.
    \end{align*}
\end{proposition}

The proof of Proposition \ref{A,B} is based on the following simple combinatorial lemma.

\begin{lemma} \label{may be easy}
    Let $\omega\in \CC$.

    $(i)$ Let $\ell$ be a non-negative integer with $n\ge \ell$. Then we have
    \begin{align*}
        \sum_{k=0}^{n-\ell}(-1)^{n-\ell-k}\binom{n+k-1}{k}\binom{n-\omega-\ell-1}{n-k-\ell}=\binom{n+\omega}{n-\ell}\,.
    \end{align*}

    $(ii)$ Let $\ell$ be a non-negative integer with $n-1\ge \ell$. Then we have
    \begin{align*}
        \sum_{k=0}^{n-1-\ell}(-1)^{n-1-k-\ell}\binom{n+k}{k}\binom{n+\omega-\ell}{n-1-k-\ell}=\binom{n-1-\omega}{n-\ell-1}\,.
    \end{align*}

    $(iii)$ We have
    \[
        \sum_{k=0}^n\binom{n-1-\omega}{k}\binom{n+\omega}{n-k}=\binom{2n-1}{n}\,.
    \]
\end{lemma}

\begin{proof}
    We first prove $(i)$. Note that we have
    \begin{align*}
        &(1-z)^{-\omega-\ell-1}=\sum_{\lambda=0}^{\infty}\binom{-\omega-\ell-1}{\lambda}(-z)^{\lambda} = \sum_{\lambda=0}^{\infty}\binom{\lambda+\omega+\ell}{\lambda}z^{\lambda}
    \end{align*}
    and
    \begin{align*}
        (1-z)^{-n}(1-z)^{n-\omega-\ell-1} &=
        \Big(\sum_{\lambda=0}^{\infty} \binom{n+\lambda-1}{\lambda}z^\lambda\Big)
        \Big(\sum_{\lambda=0}^{\infty} \binom{n-\omega-\ell-1}{\lambda}(-z)^\lambda\Big) \\
        &=\sum_{\lambda=0}^{\infty}\left(\sum_{k=0}^{\lambda}(-1)^{\lambda-k}\binom{n+k-1}{k}\binom{n-\omega-\ell-1}{\lambda-k}\right)z^{\lambda}.
    \end{align*}
    Since $(1-z)^{-\omega-\ell-1}=(1-z)^{-n}(1-z)^{n-\omega-\ell-1}$, the coefficients of $z^{n-\ell}$ in the two above formulae must be the same, hence  $(i)$. We deduce $(ii)$ from $(i)$ by replacing $(n,\ell,\omega)$ by $(n+1,\ell+2,-\omega-2)$ and $(iii)$ is just a special case of Vandermonde identity.

\end{proof}

\begin{proof}[Proof of Proposition $\ref{A,B}$]
    By the definition, we have
    \begin{align*}
        A_n(1-z)&=\sum_{k=0}^{n-1}\binom{n+k}{k}\binom{n+1/3}{n-1-k}(1-z-1)^{n-1-k}\\
        &=\sum_{k=0}^{n-1}\binom{n+k}{k}\binom{n+1/3}{n-1-k}\left[\sum_{\ell=0}^{n-1-k}\binom{n-1-k}{\ell}(-1)^{n-1-k-\ell}(1-z)^{\ell}\right]\\
        &=\sum_{\ell=0}^{n-1}\left[\sum_{k=0}^{n-1-\ell}(-1)^{n-1-k-\ell}\binom{n+k}{k}\binom{n+1/3}{n-1-k}\binom{n-1-k}{\ell}\right](1-z)^{\ell}\,.
    \end{align*}

    Combining the above with the identity
    \[
        \binom{n+1/3}{n-1-k}\binom{n-1-k}{\ell} = \binom{n+1/3}{\ell}\binom{n+1/3-\ell}{n-1-k-\ell}
    \]
    and Lemma \ref{may be easy} $(ii)$ (for $\omega=1/3$), we obtain the expected formula for $A_n(1-z)$.
    Similarly,
    \begin{align*}
        B_n(1-z)&=\sum_{k=0}^n\binom{n+k-1}{k}\binom{n-4/3}{n-k}(1-z-1)^{n-k}\\
        &=\sum_{k=0}^n\binom{n+k-1}{k}\binom{n-4/3}{n-k}\left[\sum_{\ell=0}^{n-k}\binom{n-k}{\ell}(-1)^{n-k-\ell}(1-z)^{\ell}\right]\\
        &=\sum_{\ell=0}^n\left[\sum_{k=0}^{n-\ell}(-1)^{n-k-\ell}\binom{n+k-1}{k}\binom{n-4/3}{n-k}\binom{n-k}{\ell}\right](1-z)^{\ell}\,,
    \end{align*}
    and combining the above with the identity
    \[
        \binom{n-4/3}{n-k}\binom{n-k}{\ell} = \binom{n-4/3}{n-k}\binom{n-4/3-\ell}{n-k-\ell}
    \]
    and Lemma \ref{may be easy} $(i)$ (for $\omega = 3$), we get the formula for $B_n(1-z)$. This completes the proof of Proposition~\ref{A,B}.
\end{proof}

\section{Remainder function}
\label{section: remainder}

\begin{notation}\label{Vn}

Let $A_n(1-z),B_n(1-z)$ be the polynomials defined as in Proposition \ref{A,B}.
Define the polynomial $V_n(T)$ by
\[
    V_n(T):=A_n(T^3)-TB_n(T^3)\in \Q[T].
\]
\end{notation}

Below we show that the polynomials  $A_n(T)$, $B_n(T)$ and $V_n(T)$ have the following properties, thanks to the fact they are related to Pad\'{e} approximation.

\begin{lemma} \label{ABV}

    $(i)$ For each $n \geq 0$, there exists $W_n(T)\in \Q[T]$ such that
    \[
        V_n(T) = (1-T)^{2n}W_n(T).
    \]

    $(ii)$ For any $\beta\in \C\setminus \{1\}$, we have
    \[
        A_n(\beta)B_{n+1}(\beta)-A_{n+1}(\beta)B_n(\beta)\neq 0\,.
    \]

\end{lemma}

\begin{proof}
    Let $R_n(z)$ be the series defined in Eq.~\eqref{Rn}. Since  $R_n(z)\in (1/z^{n+1})$ (see Eq.~\eqref{approximation}),
    we have
    \begin{align*}
        S_n(z):=-z^{n-1}R_n(1/z) &= -z^{n-1}[P_{n,0}(1/z)f(1/z)-P_{n,1}(1/z)] \\
        & =A_n(1-z)-(1-z)^{1/3}B_n(1-z)\in (z^{2n})\,.
    \end{align*}
    Set $g(z)=(1-z)^{1/3}$.
       Noting that $3T-3T^2+T^3 = 1-(1-T)^3$, the formal series $g(3T-3T^2+T^3)$ converges $T$-adically to $1-T$, from which we deduce the formal identity (in $\QQ[[T]]$)
    \begin{align*}
        S_n(3T-3T^2+T^3) = A_n\big((1-T)^3\big)-(1-T)B_n\big((1-T)^3\big) = V_n(1-T). 
    \end{align*}
    Since $3T-3T^2+T^3$ is divisible by $T$ and $S_n(z)\in (z)^{2n}$, we deduce that $T^{2n}$ divides the polynomial $V_n(1-T)$. This completes the proof of $(i)$.

    \smallskip

    By definition of $A_n$ and $B_n$, we have
    \begin{align*}
        A_n(1-z)B_{n+1}(1-z)-A_{n+1}(1-z)B_n(1-z)&=z^{2n}\left(P_{n,0}(1/z)P_{n+1,1}(1/z)-P_{n+1,0}(1/z)P_{n,1}(1/z)\right)\\
                                                              &=\dfrac{(n+1)_{n+1}(-1/3)_n(4/3)_n}{(n+1)!(2n)!}\cdot z^{2n}\,,
    \end{align*}
    the last identity coming from \cite[Lemma $3.7$]{KP} with $(\alpha,\gamma,\delta)=(1,-1,4/3)$. Assertion $(ii)$ follows by evaluating at $z=1-\beta$.

\end{proof}

\section{Mahler measure and estimates in the archimedean case}\label{section: Mahler}
In this section, denote by $|\cdot|$ the usual absolute value in $\CC$.
For a polynomial $P(z)\in \CC[z]$, recall the definition of \textsl{Mahler measure} $M(P)$ of $P$ given by
\[
    M(P)= \exp\Big( \int_{0}^{1} \log\big(|P(e^{2i\pi\theta}| \big)\textrm{d}\theta \Big).
\]
The \textsl{length} $L(P)$ of $P$ is the sum of the moduli of its coefficients. It is known (see for example \cite[\S $1$ ex.\,$1.2$]{wald}) that
\begin{align}
    \label{eq:formule Mahler measure and length of P}
     M(P) \leq L(P) \leq 2^{\deg P} M(P).
\end{align}
Note that for any $z\in\CC$ we have
\begin{align}
    \label{eq:formule |P(z)| and L(P)}
     |P(z)| \leq L(P)\cdot \max(1,|z|)^{\deg P}.
\end{align}

\

\begin{lemma}
    \label{lem: estimates An and Bn}
    Let $z\in\CC$. For each $n\geq 1$ we have
    \begin{align}
    \label{eq: first estimate An, Bn}
        \max\big(|A_{n}(z)|, |B_{n}(z)|  \big) \leq 4^n \max(1,|z|)^{n}\,,
            \end{align}
    and for each $n\geq 2$ we have
    \begin{align}
    \label{eq: LAn, LBn}
         L(A_{n}(z)) +  L(B_{n}(z)) \leq 4^{n}/2\,.
    \end{align}
\end{lemma}

\begin{proof}
    All the coefficients of the polynomial $A_{n}(z)$ are positive and that $A_{n}(1)$, the constant coefficient of $A_n(1-z)$, corresponds to the leading coefficient of $P_{n,1}(z)$, hence
    \[
        L(A_{n}(z)) =A_{n}(1) = \binom{2n-1}{n-1} \leq 4^{n-1} \quad (n\geq 1).
    \]
    Similarly, since all the coefficients but  the leading one of the polynomial $B_{n}(z)$ are positive, and observing that
    the constant coefficient of $B_{n}(1-z)$ equals to the leading one of $P_{n,0}(z)$, we obtain
    \begin{align}
    \label{eq proof: estimate L(B_n)}
         L(B_{n}(z)) = B_{n}(1) -2\binom{n-4/3}{n} =  \binom{2n-1}{n} -2\binom{n-4/3}{n} \leq 4^n \quad (n\geq 1).
    \end{align}
    We get Eq.~\eqref{eq: first estimate An, Bn} by using Eq.~\eqref{eq:formule |P(z)| and L(P)} with the above estimates.
    Eq.~\eqref{eq: LAn, LBn} follows from the above and by noting that in Eq.~\eqref{eq proof: estimate L(B_n)}, we can replace the upper bound $4^n$ by $4^{n-1}$ for each $n\geq 2$.
\end{proof}
\begin{lemma}\label{advantage}
Let $z\in \C$. Recall that $V_n(z)=(1-z)^{2n}\cdot W_n(z)$. We have for each $n\geq 1$,
    \begin{align}\label{eq: first estimate Vn}
       |W_n(z)| \leq 8^n \max(1,|z|)^{n+1}.
    \end{align}
\end{lemma}

\begin{proof}
Since $W_n(T)\in\Q[T]$ has degree $n+1$, Eq.~\eqref{eq:formule Mahler measure and length of P} implies that
  \begin{align*}
    2^{-n-1}L(W_n) \leq M(W_n) =  M(V_n) \leq L(V_n)  =L(A_{n}(T^3)) +  L(B_{n}(T^3)) \leq 2^{-1}\cdot 4^{n},
  \end{align*}
  for each $n\geq 2$, 
  hence $L(W_n)\leq 4\cdot 8^n$. The lemma follows by Eq.~\eqref{eq:formule |P(z)| and L(P)}.
  We have $W_1(z)=-(z^2+2z+3)/3$, thus Eq.~\eqref{eq: first estimate Vn} still holds in the case of $n=1$.
\end{proof}

\begin{remark0}
    Eq.~\eqref{approximation} gives an explicit expression of $R_n(z)$ that corresponds to the \textsl{degenerate} case of Gauss' hypergeometric function. It is possible to deduce asymptotic estimates of the remainder function $R_n(z)$ when $n$ is large enough (\textit{confer} \cite{Erd} \cite{Luke} \cite{Watson}). However, we will not use is in our proof.

    \smallskip

    Also note that we could slightly improve our inequalities (and thus Theorem~\ref{main}) by estimating the common denominators of the coefficients of approximants as Bennett does in \cite[Lemma 3.2]{BennettAus} (the proof is based on Dirichlet's prime theorem in arithmetic progressions). For simplicity, we will just combine the above argument relying on Mahler measure with properties of the local height in Section~\ref{qtt}. It will be sufficient to get our refinement.
\end{remark0}

\section{Heights and estimates in the archimedean case}\label{section:height}

Recall that $K$ is a number field of degree $m$. We denote the set of places of $K$  by ${{\mathfrak{M}}}_K$ (respectively by ${\mathfrak{M}}^{\infty}_K$ for archimedean places, and by ${{\mathfrak{M}}}^{f}_K$
for finite places).

\begin{definition}
    Given a place $v$ of $K$, we define the normalized absolute value $| \cdot |_v$ by
 \begin{align*}
&|p|_v=p^{-\tfrac{[K_v:\Q_p]}{[K:\Q]}} \ \text{if}  \ v\in{{\mathfrak{M}}}^{f}_K \ \text{and} \ v|p\enspace, \\
&|x|_v=|\sigma_v(x)|^{\tfrac{[K_v:\R]}{[K:\Q]}} \ \text{if} \ v\in {{\mathfrak{M}}}^{\infty}_K\enspace,
 \end{align*}
    where $p$ is a rational prime number in the first case and $\sigma_v$ is the embedding $K\hookrightarrow \C$ corresponding to $v$ in the second case. Following the notation of \cite{E}, for each $v\in {\mathfrak{M}_K}$ and each $\beta\in K$, we define the \textit{exponential local height} of  $\beta$ by
    \begin{align*}
         &\mathrm{H}_v({\beta})=\max\{ 1,|\beta|_v\}\,,
    \end{align*}
    and its \textit{exponential  absolute height} by
    \begin{align*}
        &\mathrm{H}({\beta})=\prod_{v\in {{\mathfrak{M}}}_K}\mathrm{H}_v({\beta})\,.
    \end{align*}
    We also set
    \begin{align*}
        s(v):=\left\{
        \begin{array}{cl}
        1/m & \ \text{if} \ v \ \text{~archimedean~real\,,} \\
        2/m & \ \text{if} \ v \ \text{~archimedean~complex\,,} \\
        0     & \ \text{if} \ v \ \text{~non-archimedean\,.}
        \end{array}
        \right.
    \end{align*}
    Given $v\in \mathfrak{M}_K$ and $\beta_1, \ldots, \beta_r\in K$, we have
    \begin{align}\label{eq:30}
        \left|\beta_1+\cdots+\beta_r\right|_v\leq r^{s(v)}\max\left(|\beta_1|_v,\ld,|\beta_r|_v\right)\,.
    \end{align}

\end{definition}

\noindent
When $v$ is archimedean, we deduce from Lemmata \ref{lem: estimates An and Bn} and \ref{advantage} the following property.
\begin{proposition}\label{archim}
    Let $v$ be an archimedean place and $\alpha\in K$. For each $n\geq 1$ we have
       \begin{align*} 
        \max\big(|A_{n}(\alpha^3)|_v, |B_{n}(\alpha^3)|_v  \big) \leq 4^{ns(v)}\mathrm{H}_v(\alpha)^{3n} \AND
        |V_n(\alpha)|_v \leq |1-\alpha|_v^{2n}\cdot 8^{ns(v)}\mathrm{H}_v(\alpha)^{n+1}\,.\\
   \end{align*}
\end{proposition}

\section{Quantitative argument}\label{qtt}

In this section, we closely follow the argument in \cite{E}. We collect the main properties we need and which are already proven in \cite{E}. We first introduce some notation used by Evertse.

Fix $\rho$ a primitive cubic root of the unity and set $L=K(\rho)$. Given $x,\ y\in U_s$ solutions of $\lam x+\mu y=1$, we put
\begin{align*}
\begin{aligned}
    \xi=\xi(x,\ y)=\lam x-\rho \mu y, \quad \eta=\eta(x,\ y)=\lam x-\rho^2\mu y, \quad \textrm{and} \quad
    \zeta=\zeta(x,\ y)=\frac{\xi(x,\ y)}{\eta(x,\ y)}.
\end{aligned}
\end{align*}

We denote by $\mathbf{1}_L$ the set of the root of the unity in $L$ and define
\begin{align}\label{thezeta}
    \mathcal{V}^0=\Big\{\zeta(x,y)\in L ~\Bigm\vert~ x, y \in U_S \quad \textrm{with} \quad \lam x+\mu y=1\AND \dfrac{\lam x}{\mu y}\notin \mathbf{1}_L\Big\}.
\end{align}
Let $T\subset \mathfrak{M}_L$ be the set of places of $L$ consisting of those lying above the places in $S$.  Put
\begin{align*}
    A=\Big(\prod_{V\in T}|3|_V\Big)^\frac{1}{2}\prod_{V\in T}|\lam\mu|_V\Big(\prod_{V\not\in T}\max(|\lam|_V,\ |\mu|_V)\Big)^3.
\end{align*}
We may suppose that $A\geq 1$ by simple properties of absolute values \cite[page 571]{E}.

\smallskip

The next two results correspond respectively to \cite[Lemma $5$]{E} and \cite[Lemma $7$]{E}.

\begin{lemma}\label{lem:5}
    Let $B$ be a real number with $\dfrac{1}{2}\leq B<1$ and set $R(B)=(1-B)^{-1}B^{\frac{B}{B-1}}$. Then, there is a set $\mca W_0$ of cardinality at most $3^sR(B)^{s-1}$, consisting of tuples $\big((\rho_V)_{V\in T}, (\Gamma_V)_{V\in T}\big)$, with  $\sum_{V\in T} \Gamma_V=B$ and such that $\rho_V^3=1$ and  $\Gamma_V\geq 0$ for each $V\in T$, which satisfies the following property. For each $\zeta\in\mca V^0$, there exists $\big((\rho_V)_{V\in T}, (\Gamma_V)_{V\in T}\big)\in \mca W^0$ with
    \begin{align}\label{eq:49}
        \min(1,\ |1-\rho_V\zeta|_V)\leq\left(\dfrac{8A}{H(\zeta)^{3}}\right)^{\Gamma_V} \quad \textrm{for } V\in T.
    \end{align}
\end{lemma}

In the following we fix a tuple $\big((\rho_V)_{V\in T}, (\Gamma_V)_{V\in T}\big)\in \mca W^0$, where $\mca W^0$ is as in Lemma~\ref{lem:5}.

\begin{lemma}\label{lem:7}
    Let $B$ be a real number with $\dfrac{2}{3}<B<1$, and let $\zeta_1,\ \cdots,\ \zeta_{k+1}\in \mca V^0$ be distinct elements
    with $H(\zeta_1)\leq\cdots\leq H(\zeta_{k+1})$ and which satisfy inequality
    Eq.~\eqref{eq:49}. Then we have
    \begin{align}\label{eq:60}
        H(\zeta_{k+1})\geq\Big(\frac{A^{1-B}}{6\times8^B}\Big)^{\frac{(3B-1)^k-1}{3B-2}}H(\zeta_1)^{(3B-1)^k}.
    \end{align}
\end{lemma}

The lemma below is an improved version of \cite[Lemma $8$]{E}.

\begin{lemma}\label{lem:8}
    Let $B$ be a real number with $\dfrac{5}{6}<B<1$. Let $r_0$ and $k$ be integers with
    \[
        r_0>\dfrac{6+3B}{3B(6B-5)} \AND (3B-1)^{k+1}>3r_0+4.
    \]
    Define five exponents as follows
    \begin{align*}
        f_1(B,\ r_0)&=\frac{2r_0B(3B-1)+B}{3r_0B(6B-5)-6-3B},\\
        f_2(B,\ r_0)&=\frac{3r_0B+3}{3r_0B(6B-5)-6-3B},\\
        g_1(B,\ k,\ r_0)&=\frac{B-(1-B)(3B-1)\frac{(3B-1)^k-1}{3B-2}}{(3B-1)^{k+1}-3r_0-4},\\
        g_2(B,\ k,\ r_0)&=\frac{r_0+1+(3B-1)\frac{(3B-1)^k-1}{3B-2}}{(3B-1)^{k+1}-3r_0-4}, \\
        g_3(B,\ k,\ r_0)&=\frac{r_0}{(3B-1)^{k+1}-3r_0-4}.
    \end{align*}

    Then there exist at most $k$ elements $\zeta\in \mca V^0$ satisfying Eq.~\eqref{eq:49} and
    \begin{align}\label{eq:61}
        H(\zeta)>\max\Big((8A)^{f_1(B,\ r_0)}(24)^{f_2(B,\ r_0)},\ (8A)^{g_1(B,\ k,\ r_0)}(48)^{g_2(B,\ k,\ r_0)}2^{-g_3(B,k,r_0)}\Big).
    \end{align}
\end{lemma}

\begin{remark0}
     The structure of the proof is the same as that of \cite[Lemma $8$]{E}, but we will use sharper estimates. It allows us to replace the constant $96\sqrt 3$ of Evertse by $24$ and $48$ (in the first and second term of the maximum respectively) and to get the extra term $2^{-g_3(B,k,r_0)}$. For the numerical computation of Section~\ref{section: numerical opt}, the maximum in Eq.~\eqref{eq:61} is the term $(8A)^{f_1(B,\ r_0)}(24)^{f_2(B,\ r_0)}$.
\end{remark0}

\begin{proof}
    By contradiction, suppose that there exist $k+1$ elements $\zeta_1,\dots,\zeta_{k+1}\in \mca V^0$ satisfying $(\ref{eq:49})$ and $(\ref{eq:61})$. For simplicity, write $h_i=H(\zeta_i)$ for $i=1,\dots,k+1$. We may assume that $h_1\leq h_3\leq h_4\cdots\leq h_{k+1}\leq h_2$. Note that the function $r\mapsto f(B,r)$ is decreasing on $[r_0,\infty[$ and tends to $2(3B-1)/(3(6B-5))> 4/3$. We therefore have
    \begin{align}\label{eq:62}
        h_2\geq h_1\geq \big(8A\big)^{f_1(B,r_0)}\geq \big(8A\big)^\frac{4}{3}>1.
    \end{align}
    In particular, $\zeta_1$ is not a root of unity since $h_1>1$.

    \smallskip

    We claim that there exists an integer $\ell\geq r_0+1$ such that
    \begin{align}\label{eq:63}
        2\cdot 24^\ell\big(8A\big)^Bh_1^{3\ell+1}<h_2^{3B-1}\leq 2\cdot 24^{\ell+1}\big(8A\big)^Bh_1^{3\ell+4}.
    \end{align}
    To prove this, it is sufficient to show that the left-hand inequality holds for $\ell=r_0+1$. According to Lemma~\ref{lem:7}, it is even sufficient to show that,
    \[
        2^{-r_0}\cdot 48^{r_0+1}\big(8A\big)^Bh_1^{3r_0+4} < \bigg(\frac{(8A)^{1-B}}{48}\bigg)^{(3B-1)\frac{(3B-1)^k-1}{3B-2}}h_1^{(3B-1)^{k+1}},
    \]
    which, by definition of $g_1$, $g_2$ and $g_3$, is equivalent to $(8A)^{g_1(B,k,r_0)}(48)^{g_2(B,k,r_0)}2^{-g_3(B,k,r_0)} < h_1$. This last inequality holds by Eq.~\eqref{eq:61}, hence our claim.

    \smallskip

    For each integer $n$, define the polynomials $\A_n(z):=3^nA_n(z)$, $\B_n(z):=3^n B_n(z)$ as well as
     \[
        \V_n(z):=\A_n(z)-z\B_n(z) = 3^nV_n(z),
     \]
      where $A_n(z)$ and $B_n(z)$ are as in Proposition~\ref{A,B}. Note that $\A_n(z)$ and $\B_n(z)$ have integer coefficients. In particular, for each finite place $v\in T$ and each $\zeta\in K$, we have $|\A_n(\zeta^3)|_v, |\B_n(\zeta^3)|_v \leq H_v(\zeta)^{3n}$. Combined with Proposition~\ref{archim}, we deduce that for each place $v\in T$ (finite or archimedean) and each $\zeta\in K$, we have
    \begin{align*}
        &\max\big(|\A_{n}(\zeta^3)|_v, |\B_{n}(\zeta^3)|_v  \big) \leq 12^{ns(v)}\mathrm{H}_v(\zeta)^{3n} \leq 24^{ns(v)}\mathrm{H}_v(\zeta)^{3n+1},\\
        &|\V_n(\zeta)|_v \leq |1-\zeta|_v^{2n}\cdot 24^{ns(v)}\mathrm{H}_v(\zeta)^{n+1}\leq |1-\zeta|_v^{2n}\cdot 24^{ns(v)}\mathrm{H}_v(\zeta)^{3n+1}\,.
    \end{align*}
    Now, for each integer $m$, set
    \[
        U_n=\zeta_2 \A_n\left({\zeta_1}^3\right)-\zeta_1 \B_n\left({\zeta_1}^3\right).
    \]
    Note that by  Lemma \ref{ABV} (ii) (and since $\zeta_1\neq 1$) we have $(U_{m-1},U_m)\neq (0,0)$ for each integer $m\geq 1$.
    Define $n$ by $n=\ell$ if $U_{\ell}\neq 0$, and $n=\ell-1$ if $U_{\ell}=0$. By the above we have $U_n\neq 0$.

    \smallskip

    Then, using our estimates of $\A_n$, $\B_n$ and $\V_n$, given $V\in T$, we find

    \begin{align*}
        |U_n|_V =\big|\rho_V\zeta_2\A_n(\zeta_1^3)-\rho_V\zeta_1\B_n(\zeta_1^3)\big|_V &=|(\rho_V\zeta_2-1)\A_n({\zeta_1}^3)+\V_n(\rho_V\zeta_1)|_V\\
        &\leq 2^{s(V)}\max\Big(|1-\rho_V\zeta_2|_V,\ |1-\rho_V\zeta_1|_V^{2n}\Big) \times 24^{ns(V)}\mathrm{H}_V(\zeta_1)^{3n+1}\\
        &\leq 2^{s(V)}\max\Big(8Ah_2^{-3},\ (8A)^{2n} h_1^{-6n}\Big)^{\Gamma_V} 24^{ns(V)}\mathrm{H}_V(\zeta_1)^{3n+1},
    \end{align*}
    where the last inequality follows from Eq.~\eqref{eq:49}.

    \smallskip

    Given $V\not\in T$, we simply use the triangle inequality to get the estimate
    \begin{align*}
        |U_n|_V&\leq 2^{s(V)}\max\left(|\zeta_2|_V|A_n({\zeta_1}^3)|_V,\ |\zeta_1|_V|B_n({\zeta_1}^3)|_V\right) \leq 2^{s(V)}
        24^{ns(V)}H_V(\zeta_1)^{3n+1}H_V(\zeta_2).
    \end{align*}
    Combining all those estimates with the product formula and the identities $\sum_{T\in V}\Gamma_T = B$ and $\sum_{v}s(v) = 1$ (and since $\prod_{v\in{{\mathfrak{M}}}_K} H_v(\zeta) = H(\zeta)$ by definition of the height), we obtain
    \begin{align}
        1 = \prod_{V\in T}|U_n|_V \times \prod_{V\notin T}|U_n|_V \leq 2\cdot 24^n\times \max\Big((8A)^Bh_2^{1-3B}h_1^{3n+1},\ (8A)^{2nB} h_2h_1^{3n+1-6nB}\Big).
    \end{align}

    First, note that by Eq.~\eqref{eq:63} with $n\leq \ell$, we have $2\cdot 24^n(8A)^Bh_2^{1-3B}h_1^{3n+1} < 1$. Combined with the above, we deduce that
    \[
        1 \leq 2\cdot 24^n\times (8A)^{2nB} h_2h_1^{3n+1-6nB}.
    \]
    Rising the above inequalities to the power $3B-1$ and using the right-hand side of Eq.~\eqref{eq:63} together with $\ell \leq n+1$, we obtain
    \begin{align*}
        1 & \leq 2^{3B-1}24^{(3B-1)n}(8A)^{2nB(3B-1)}h_1^{(3B-1)(3n+1-6nB)}h_2^{3B-1}\\
        &\leq 2^{3B}24^{(3B-1)n+n+2}(8A)^{2nB(3B-1)+B}h_1^{(3B-1)(3n+1-6nB)+3n+7}\\
        & < (8A)^{2nB(3B-1)+B}(24)^{3nB+3}h_1^{6+3B-3nB(6B-5)},
    \end{align*}
    hence $h_1 < (8A)^{f_1(B,n)}(24)^{f_2(B,n)}$. Since $f_1(B, r)$, $f_2(B, r)$ are positive and decreasing in $r$ on $r\geq r_0$ and since $r_0\leq n$ this contradicts Eq.~\eqref{eq:61}. Consequently, the number of $\zeta\in \mca V^0$ satisfying Eq.~\eqref{eq:61} is at most $k$.

\end{proof}
\section{Numerical optimization}
\label{section: numerical opt}

We adapt the strategy in \cite[Section 8]{E} to estimate the number of solutions $(x,y)$ of the $S$-unit equation Eq.~\eqref{unitequation}. Since such a couple $(x,y)$ is entirely determined by the corresponding $\zeta(x,y)$, the first step is to count the elements $\zeta\in \VV$ which satisfies a fixed system Eq.~\eqref{eq:49}. We call $f$-part and $g$-part of Eq.~\eqref{eq:61} the quantities $(8A)^{f_1(B,\ r_0)}(24)^{f_2(B,\ r_0)}$ and $(8A)^{g_1(B,\ k,\ r_0)}(48)^{g_2(B,\ k,\ r_0)}2^{-g_3(B,k,r_0)}$ respectively. Fix $C>0$ (we will specify this parameter later).

\smallskip

\noindent\textbf{Step 1.}  We choose parameters $B,r,k$ satisfying the conditions of Lemma~\ref{lem:8} and we compute the associated quantities $f_1$, $f_2$, $g_1$, $g_2$ and $g_3$. Generally, with our choices we have $f$-part $>$ $g$-part, so that the maximum in Eq.~\eqref{eq:61} is given by the $f$-part. Then, by Lemma~\ref{lem:8} there are at most $k$ numbers $\zeta\in\VV$ satisfying Eq.~\eqref{eq:49} with $H(\zeta) > \textrm{$f$-part}$.

\smallskip

\noindent\textbf{Step 2.} We estimate the number $t$ of $\zeta\in\VV$ satisfying a fixed Eq.~\eqref{eq:49} with $C \leq H(\zeta) \leq \textrm{$f$-part}$. For that, it suffices to apply Lemma~\ref{lem:7} for those numbers $\zeta_1,\dots,\zeta_t$. Assuming $H(\zeta_1)\leq \cdots \leq H(\zeta_t)$ and combining Eq.~\eqref{eq:60} with our hypothesis $H(\zeta_t) \leq \textrm{$f$-part}$ and $H(\zeta_1) \geq C$, we obtain
\begin{align*}
  \Big(\frac{A^{1-B}}{6\times8^B}\Big)^{1/(3B-2)} \times \textrm{$f$-part} \geq \Big(\Big(\frac{A^{1-B}}{6\times8^B}\Big)^{1/(3B-2)}\times C\Big)^{(3B-1)^{t-1}}.
\end{align*}
We deduce that
\begin{align*}
    (3B-1)^{t-1} \leq \frac{a + b\log A}{c + d\log A} \qquad \textrm{where} \quad
    \left\{ \begin{array}{l}
      a = \log \Big((6\cdot 8^B)^{-1/(3B-2)}8^{f_1} 24^{f_2} \Big)  \\
      b = \displaystyle \frac{1-B}{3B-2} + f_1 \\
      c = \log \Big( C (6\cdot 8^B)^{-1/(3B-2)} \Big)\\
      d = \displaystyle \frac{1-B}{3B-2}.
    \end{array}\right.
\end{align*}
Our choice of $B,r,C$ will imply that $a,b,c,d > 0$. Since $A\geq 0$ and since the function $x \mapsto (a+bx) / (c + dx)$ is increasing (resp. decreasing) on $[0,+\infty)$ if $a/c < b/d$ (resp.if $a/c> b/d$), the quantity $(3B-1)^{t-1}$ is bounded from above by $\max\{a/c,b/d\}$. Hence an upper bound $N$ (which does not depend on $A$) for $t$.

\smallskip

\noindent\textbf{Step 3.} There are therefore at most $N+k$ numbers $\zeta\in\VV$ satisfying a fixed system Eq.~\eqref{eq:49} with $H(\zeta) \geq C$. Moreover, according to Lemma~\ref{lem:5}, we may assume that the given system Eq.~\eqref{eq:49} is one of the $3^sR(B)^{s-1}$ ones associated to the elements of $\mca W_0$. On the other hand, Evertse proved (see \cite[Eq. (36)]{E}) that for any solution $(x,y)$ of Eq.~\eqref{unitequation}, we have
\[
  H\Big(\frac{\lambda x}{\mu y}\Big) \leq 2 H\big(\zeta(x,y)\big).
\]
Since each solution $(x,y)$ of Eq.~\eqref{unitequation} is completely determined by $\zeta(x,y)$, we deduce from the above that the number of solutions $(x,y)$ of the $S$-unit equation Eq.~\eqref{unitequation}  with $H(\lambda x / \mu y) \geq 2C$ is at most $3^s(N+k)R(B)^{s-1}$.

\smallskip

\noindent\textbf{Step 4.} It remains to estimate the number of solutions $(x,y)$ with  $H(\lambda x / \mu y) < 2C$. Studying the group structure of the $S$-units, Evertse shows that given a non-zero $n\in\N$, there exists at most $5n^s(2(2C)^{3/n})^m$ solutions $(x,y)$ of Eq.~\eqref{unitequation} with $H(\lambda x /\mu y) < 2C$. Thus, the total number of solutions of the $S$-unit equation is at most
\[
    5n^s(2(2C)^{3/n})^m + 3^s(N+k)R(B)^{s-1}.
\]

\smallskip

\noindent\textbf{Choice I}\\
Take $B=0.834$, $r_0=1600$ and $C = \textrm{e}^{7.8}/2$. Then the condition on $k$ is $k>19.8389\cdots$, hence we choose $k=20$. A short computation gives
    \begin{align*}
    & f_1(B,\ r_0)=533.814\cdots & f_2(B,\ r_0)=533.391\cdots & ~ \\
    & g_1(B,\ k,\ r_0)=-5.20814\cdots & {g_2}(B,\ k,\ r_0)=36.3095\cdots  & \quad g_3(B,\ k,\ r_0)=4.91666\cdots
\end{align*}
We have $f$-part$=e^{2805.183\cdots} \times A^{533.814\cdots}$, and $g$-part$=e^{126.323\cdots} \times A^{-5.20814\cdots}$. Since $A\geq1$, the maximum
in Eq.~\eqref{eq:61} is given by the $f$-part. A short computation show that the upper bound $N$ coming from Step $2$ is $N = 26$, so that $k + N = 46$. Finally, in the last step we choose $n = 45$. With this choice, the total number of solutions of the $S$-unit equation is at most
\[
    5\left(3.36406\right)^{m}\cdot 45^s + 3.06759\cdot \left(44.9866\right)^s.
\]

\noindent\textbf{Choice II}\\
Take $B=0.84$, $r_0=100$ and $C = \textrm{e}^{7.5}/2$. Then the condition on $k$ is $k>12.6539\cdots$, so we choose $k=13$. A short computation gives
\begin{align*}
    &f_1(B,\ r_0)=164.230\cdots &f_2(B,\ r_0)=163.461\cdots &\\
    &g_1(B,\ k,\ r_0)=-2.25320\cdots  &{g_2}(B,\ k,\ r_0)=16.3237\cdots,   & \quad g_3(B,\ k,\ r_0)=2.1093\cdots
\end{align*}
Once again we find $f$-part $>$ $g$-part. The upper bound $N$ coming from Step $2$ is $N = 31$, so that $k + N = 44$. In step $4$ we choose $n = 47$. The total number of solutions of the $S$-unit equation is at most
\[
    5\times 3.22803^{m}\cdot 47^s + 2.81864\times 46.8312^s.
\]

Combining the two upper bounds together, we get Eq.~\eqref{winner} and Theorem~\ref{main} follows.

\qed

\noindent{\bf Acknowledgments}\\
The third author was supported by JSPS Postdoctoral Fellowship No.~PE20746 and Research Support Allowance Grant for JSPS Fellows,
during his stay 2021-2022 in Japan.
The first author is also supported by JSPS KAKENHI Grant no. 18K03225 and Grant no. 21K03171. 
Thanks to these Fellowship and Grants,  the authors could meet to work together.

\bibliography{}

\begin{scriptsize}
\begin{minipage}[t]{0.33\textwidth}

\noindent
Noriko Hirata-Kohno,
\\hiratakohno.noriko@nihon-u.ac.jp
(hirata@math.cst.nihon-u.ac.jp),\\
\& Yukiko Washio,
\\washio.yukiko@nihon-u.ac.jp,
\\Department of Mathematics, \\College of Science \& Technology, \\Nihon University,
\\Kanda, Chiyoda, Tokyo, \\101-8308, Japan\\
\end{minipage}
\begin{minipage}[t]{0.33\textwidth}
Makoto Kawashima,
\\kawashima.makoto@nihon-u.ac.jp,
\\Department of Liberal Arts \\and Basic Sciences, \\College of Industrial Engineering, \\Nihon University,
Izumi-chou, \\Narashino, Chiba, 275-8575, Japan\\\\
\end{minipage}
\begin{minipage}[t]{0.35\textwidth}

Anthony Po\"{e}ls,\\
poels@math.univ-lyon1.fr,\\
Universit\'e Claude Bernard Lyon 1,\\
UMR~5208, 
\\Institut Camille Jordan, \\
F-69622 Villeurbanne Cedex, France
\end{minipage}

\end{scriptsize}

\end{document}